\documentclass[final]{siamltex}

\usepackage{graphicx}
\usepackage{color}
\usepackage{bm}
\usepackage{amsfonts}
\usepackage{amsmath}
\usepackage{multirow}

\newcommand{\tr}{^{\sf T}}
\newcommand{\m}[1]{{\bf{#1}}}
\newcommand{\g}[1]{\bm #1}
\newcommand{\C}[1]{{\cal {#1}}}

\title{
Convergence rate for a Gauss collocation method
applied to unconstrained optimal control 
\thanks{
June 1, 2015, revised August 14, 2015
The authors gratefully acknowledge support by
the Office of Naval Research under grants N00014-11-1-0068 and
N00014-15-1-2048, and by the National Science Foundation under
grants DMS-1522629 and CBET-1404767.}
}

\author{
William W. Hager\thanks{{\tt hager@ufl.edu},
http://people.clas.ufl.edu/hager/,
PO Box 118105,
Department of Mathematics,
University of Florida, Gainesville, FL 32611-8105.
Phone (352) 294-2308. Fax (352) 392-8357.}
\and
Hongyan Hou\thanks{{\tt hongyan388@gmail.com},
        Chemical Engineering,
        Carnegie Mellon University,
        5000 Forbes Avenue, Pittsburgh, PA 15213.}
\and
Anil V. Rao\thanks{{\tt anilvrao@ufl.edu},
        http://www.mae.ufl.edu/rao
        Department of Mechanical and Aerospace Engineering,
        P.O. Box 116250, Gainesville, FL 32611-6250.
        Phone:(352) 392-0961. Fax:(352) 392-7303}
}

\begin{document}
\maketitle

\begin{abstract}
A local convergence rate is established for an orthogonal collocation
method based on Gauss quadrature applied to an unconstrained optimal
control problem.
If the continuous problem has a sufficiently smooth solution
and the Hamiltonian satisfies a strong convexity condition,
then the discrete problem possesses a local minimizer
in a neighborhood of the continuous solution, and as the number
of collocation points increases, the discrete solution convergences
exponentially fast in the sup-norm to the continuous solution.
This is the first convergence rate result for an orthogonal collocation
method based on global polynomials applied to an optimal control problem.
\end{abstract}

\begin{keywords}
Gauss collocation method, convergence rate, optimal control,
orthogonal collocation
\end{keywords}

\pagestyle{myheadings} \thispagestyle{plain}
\markboth{W. W. HAGER, H. HOU, AND A. V. RAO}
{GAUSS COLLOCATION}

\section{Introduction}
A convergence rate is established for an orthogonal collocation method
applied to an unconstrained control problem of the form
\begin{equation}\label{P}
\begin{array}{clc}
\mbox {minimize} &C(\m{x}(1))&\\
\mbox {subject to} &\dot{\m{x}}(t)=
\m{f(x}(t), \m{u}(t)),&t\in[-1,1],\\
&\m{x}(-1)=\m{x}_0,&
\end{array}
\end{equation}
where the state ${\m x}(t)\in \mathbb{R}^n$, 
$\dot{\m x}\equiv\displaystyle\frac{d}{dt}{\m x}$,
the control ${\m u}(t)\in {\mathbb R}^m$, 
${\m f}: {\mathbb R}^n \times {\mathbb R}^m\rightarrow {\mathbb R}^n$, 
$C: {\mathbb R}^n \rightarrow {\mathbb R}$,
and ${\m x}_0$ is the initial condition, which we assume is given.
Assuming the dynamics $\dot{\m{x}}(t)= \m{f(x}(t), \m{u}(t))$ is nice enough,
we can solve for the state $\m{x}$ as a function of the control $\m{u}$,
and the control problem reduces to an unconstrained minimization over $\m{u}$.

Let $\C{P}_N$ denote the space of polynomials of degree at most $N$
defined on the interval $[-1, +1]$, and let $\C{P}_N^n$ denote the
$n$-fold Cartesian product $\C{P}_N \times \ldots \times \C{P}_N$.
We analyze a discrete approximation to (\ref{P}) of the form
\begin{equation}\label{D}
\begin{array}{cl}
\mbox {minimize} &C(\m{x}(1))\\
\mbox {subject to} &\dot{\m{x}}(\tau_i)=
\m{f(x}(\tau_i), \m{u}_i),\quad 1 \le i \le N,\\
&\m{x}(-1)=\m{x}_0, \quad \m{x} \in \C{P}_N^n.
\end{array}
\end{equation}
The collocation points $\tau_i$, $1 \le i \le N$, are where the
equation should be satisfied, and $\m{u}_i$ is the control approximation
at time $\tau_i$.
The dimension of $\C{P}_N$ is $N+1$, while there are
$N+1$ equations in (\ref{D}) corresponding to the
collocated dynamics at $N$ points and the initial condition.
When the discrete dynamics is nice enough,
we can solve for the discrete state $\m{x} \in \C{P}_N^n$ as a function
of the discrete controls $\m{u}_i$, $1 \le i \le N$, and the discrete
approximation reduces to an unconstrained minimization over the discrete
controls.

We analyze the method developed in \cite{Benson2,GargHagerRao10a}
where the collocation points are the Gauss quadrature
abscissas, or equivalently, the roots of a Legendre polynomial.
Other sets of collocation points that have been studied include
the Lobatto quadrature points \cite{Elnagar1,Fahroo2},
the Chebyshev quadrature points \cite{Elnagar4, FahrooRoss02},
the Radau quadrature points
\cite{Fahroo3,GargHagerRao11a,LiuHagerRao15, PattersonHagerRao14},
and extrema of Jacobi polynomials \cite{Williams1}.
The Gauss quadrature points that we analyze are symmetric about
$t = 0$ and satisfy
\[
-1 < \tau_1 < \tau_2 < \ldots < \tau_N < +1 .
\]
In addition, we employ two noncollocated points $\tau_0 = -1$ and
$\tau_{N+1} = +1$.

Our goal is to show that if $(\m{x}^*, \m{u}^*)$ is a local minimizer
for (\ref{P}), then the discrete problem (\ref{D}) has a local minimizer
$(\m{x}^N, \m{u}^N)$ that converges exponentially fast in $N$ to
$(\m{x}^*, \m{u}^*)$ at the collocation points.
This is the first convergence rate result for an orthogonal collocation
method based on global polynomials applied to an optimal control problem.
A consistency result for a scheme based on global polynomials and
Lobatto collocation is given in \cite{GongRossKangFahroo08}.
Convergence rates have been obtained previously when the approximating
space consists of piecewise polynomials as in
\cite{DontchevHager93, DontchevHager97,DontchevHagerVeliov00,
DontchevHagerMalanowski00, Hager99c, Kameswaran1, Reddien79}.
In these results,
convergence is achieved by letting the mesh spacing tend to zero.
In our results, on the other hand, convergence is achieved by letting
$N$, the degree of the approximating polynomials, tend to infinity.

To state our convergence results in a precise way,
we need to introduce a function space setting.
Let $\C{C}^k (\mathbb{R}^n)$ denote the space of $k$ times
continuously differentiable functions
$\m{x}: [-1, +1] \rightarrow \mathbb{R}^n$ with the sup-norm
$\| \cdot \|_\infty$ given by
\begin{equation}\label{csup}
\|\m{x}\|_\infty = \sup \{ |\m{x} (t)| : t \in [-1, +1] \} ,
\end{equation}
where $| \cdot |$ is the Euclidean norm.
It is assumed that (\ref{P}) has a local minimizer
$(\m{x}^*, \m{u}^*)$ in $\C{C}^1 (\mathbb{R}^n) \times \C{C}^0 (\mathbb{R}^m)$.
Given $\m{y} \in \mathbb{R}^n$, the ball with center $\m{y}$ and radius
$\rho$ is denoted
\[
\C{B}_\rho (\m{y}) = \{ \m{x} \in \mathbb{R}^n :
|\m{x} - \m{y}| \le \rho \} .
\]
It is assumed that there exists an open set
$\Omega \subset \mathbb{R}^{m+n}$ and $\rho > 0$ such that
\[
\C{B}_\rho (\m{x}^*(t),\m{u}^*(t)) \subset \Omega \mbox{ for all }
t \in [-1, +1].
\]
Moreover, the first two derivative of $f$ and $C$ are
continuous on the closure of
$\Omega$ and on $\C{B}_\rho (\m{x}^*(1))$ respectively.

Let $\g{\lambda}^*$ denote the solution of the linear costate equation
\begin{equation}\label{costate}
\dot{\g{\lambda}}^*(t)=-\nabla_xH({\m x}^*(t), {\m u}^*(t), {\g \lambda}^*(t)),
\quad {\g \lambda}^*(1)=\nabla C({\m x}^*(1)),
\end{equation}
where $H$ is the Hamiltonian defined by 
$H({\m x}, {\m u}, {\g \lambda}) ={\g\lambda}\tr {\m f}({\m x}, {\m u})$.
Here $\nabla C$ denotes the gradient of $C$.
By the first-order optimality conditions (Pontryagin's minimum principle),
we have
\begin{equation} \label{controlmin}
\nabla_u H({\m x}^*(t), {\m u}^*(t), {\g \lambda}^*(t)) = \m{0}
\end{equation}
for all $t \in [-1, +1]$.

Since the discrete collocation problem (\ref{D}) is finite dimensional,
the first-order optimality conditions (Karush-Kuhn-Tucker conditions)
imply that when a constraint qualification holds
\cite{NocedalWright2006}, the gradient of the Lagrangian vanishes.
By the analysis in \cite{GargHagerRao10a}, the gradient of the
Lagrangian vanishes if and only if there exists
$\g{\lambda} \in \C{P}_N^n$ such that
\begin{eqnarray}
\dot{\g \lambda}(\tau_i) &=&
-\nabla_x H\left({\m x} (\tau_i),{\m u}_i, {\g \lambda} (\tau_i) \right), 
\quad 1 \leq i \leq N,  \label{dcostate} \\
{\g \lambda}(+1) &=& \nabla C(\m{x}(+1)), \label{dterminal}\\
\m{0} &=&
\nabla_u H\left({\m x}(\tau_i),{\m u}_i, {\g \lambda} (\tau_i) \right), 
\quad 1\leq i\leq N. \label{dcontrolmin}
\end{eqnarray}

The assumptions that play a key role in the convergence
analysis are the following:

\begin{itemize}
\item[(A1)]
$\m{x}^*$ and $\g{\lambda}^* \in \C{C}^{\eta+1}$ for some $\eta \ge 3$.
\item[(A2)]
For some $\alpha > 0$,
the smallest eigenvalue of the Hessian matrices
\[
\nabla^2 C(\m{x}^*(1)) \quad \mbox{and} \quad
\nabla^2_{(x,u)} H(\m{x}^* (t),
\m{u}^* (t), \g{\lambda}^* (t) )
\]
is greater than $\alpha$, uniformly for $t \in [-1, +1]$.
\item[(A3)]
The Jacobian of the dynamics satisfies
\[
\|\nabla_x \m{f} (\m{x}^*(t), \m{u}^*(t))\|_\infty \le 1/4
\quad \mbox{and} \quad
\|\nabla_x \m{f} (\m{x}^*(t), \m{u}^*(t))\tr\|_\infty \le 1/4
\]
for all $t \in [-1, +1]$ where $\| \cdot \|_\infty$ is the matrix
sup-norm (largest absolute row sum), and the Jacobian
$\nabla_x \m{f}$ is an $n$ by $n$  matrix whose $i$-th row is
$(\nabla_x f_i)\tr$.
\end{itemize}
\smallskip

The smoothness assumption (A1) is used to obtain a bound for the
accuracy with which the interpolant of the continuous state $\m{x}^*$
satisfies the discrete dynamics.
The coercivity assumption (A2) ensures that the solution of the
discrete problem is a local minimizer.
The condition (A3) does not appear in convergence analysis for
(local) piecewise polynomial techniques
\cite{DontchevHager93, DontchevHager97,DontchevHagerVeliov00,
DontchevHagerMalanowski00, Hager99c, Kameswaran1, Reddien79}.
It arises when we approximate a solution by polynomials defined
on the entire interval $[-1, +1]$.
More precisely, in the analysis, the dynamics is linearized around
$(\m{x}^*, \m{u}^*)$, and
(A3) implies that when we perturb the linearized dynamics,
the state perturbation is bounded uniformly in $N$ with respect
to the perturbation in the dynamics.
If the domain $[-1, +1]$ is partitioned into uniform subdomains
of width $h$ and a different polynomial is used on each subdomain,
then (A3) is replaced by
\[
\|\nabla_x \m{f} (\m{x}^*(t), \m{u}^*(t))\|_\infty \le 1/(2h)
\quad \mbox{and} \quad
\|\nabla_x \m{f} (\m{x}^*(t), \m{u}^*(t))\tr\|_\infty \le 1/(2h)
\]
which is satisfied when $h$ is sufficiently small.
In general, (A3) could be replaced by any condition that ensures
stability of the linearized dynamics.

In addition to the 3 assumptions, the analysis employs 2 properties
of the Gauss collocation scheme.
Let $\omega_j$, $1 \le j \le N$, denote the Gauss quadrature weights, and
for $1 \le i \le N$ and $0 \le j \le N$, define
\begin{equation}\label{Ddef}
D_{ij} = \dot{L}_j (\tau_i), \quad \mbox{where }
L_j (\tau) := \prod^N_{\substack{i=0\\i\neq j}}
\frac{\tau-\tau_i}{\tau_j-\tau_i}.
\end{equation}
$\m{D}$ is a differentiation matrix in the sense that
$(\m{Dp})_i = \dot{p} (\tau_i)$, $1 \le i \le N$,
where $p \in \C{P}_N$ is the polynomial that satisfies
$p(\tau_j) = p_j$ for $0 \le j \le N$.
The submatrix $\m{D}_{1:N}$ consisting of the tailing $N$ columns of $\m{D}$,
has the following properties:
\smallskip
\begin{itemize}
\item [(P1)]
$\m{D}_{1:N}$ is invertible and
$\| \m{D}_{1:N}^{-1}\|_\infty \le 2$.
\item [(P2)]
If $\m{W}$ is the diagonal matrix containing the quadrature weights $\g{\omega}$
on the diagonal, then the rows of the
matrix $[\m{W}^{1/2} \m{D}_{1:N}]^{-1}$ have Euclidean norm bounded by
$\sqrt{2}$.
\end{itemize}
\smallskip
The fact that $\m{D}_{1:N}$ is invertible is established in
\cite[Prop. 1]{GargHagerRao10a}.
The bounds on the norms in (P1) and (P2), however, are more subtle.
We refer to (P1) and (P2) as properties rather than assumptions since
the matrices are readily evaluated, and we can check numerically that
(P1) and (P2) are always satisfied.
In fact, numerically we find that
$\| \m{D}_{1:N}^{-1}\|_\infty = 1 + \tau_N$ where the last
Gauss quadrature abscissa $\tau_N$ approaches $+1$ as $N$ tends to $\infty$.
On the other hand, we do not yet have a general proof for the properties
(P1) and (P2).
In contrast, conditions (A1)--(A3) are assumptions that are only satisfied
by certain control problems.

If $\m{x}^N \in \C{P}_N^n$ is a solution of (\ref{D}) associated
with the discrete controls $\m{u}_i$, $1 \le i \le N$, and
if $\g{\lambda}^N \in \C{P}_N^n$ satisfies
(\ref{dcostate})--(\ref{dcontrolmin}), then we define
\[
\begin{array}{llllllll}
\m{X}^N &= [ &\m{x}^N(-1), & \m{x}^N(\tau_1), & \ldots,
& \m{x}^N(\tau_N), & \m{x}^N(+1) &], \\
\m{X}^* &= [ &\m{x}^*(-1), & \m{x}^*(\tau_1), & \ldots,
& \m{x}^*(\tau_N), & \m{x}^*(+1) &], \\
\m{U}^N &= [ && \m{u}_1, & \ldots, & \m{u}_N\; & &], \\
\m{U}^* &= [ && \m{u}^*(\tau_1), & \ldots, & \m{u}^*(\tau_N)& & ], \\
\g{\Lambda}^N &= [ &\g{\lambda}^N(-1), & \g{\lambda}^N(\tau_1),
& \ldots, & \g{\lambda}^N(\tau_N), & \g{\lambda}^N(+1) &], \\
\g{\Lambda}^* &= [ &\g{\lambda}^*(-1), & \g{\lambda}^*(\tau_1),
& \ldots, & \g{\lambda}^*(\tau_N), & \g{\lambda}^*(+1) &].
\end{array}
\]
For any of the discrete variables, we define a discrete
sup-norm analogous to the continuous sup-norm in (\ref{csup}).
For example, if $\m{U}^N \in \mathbb{R}^{mN}$ with $\m{U}_i \in \mathbb{R}^m$,
then
\[
\|\m{U}^N\|_\infty = \sup \{ |\m{U}_i| : 1 \le i \le N \}.
\]
The following convergence result is established:
\smallskip
\begin{theorem}\label{maintheorem}
If $(\m{x}^*, \m{u}^*)$ is a local minimizer for the continuous problem
$(\ref{P})$ and both {\rm (A1)--(A3)} and {\rm (P1)--(P2)} hold,
then for $N$ sufficiently large with $N > \eta+1$,
the discrete problem $(\ref{D})$ has a
local minimizer $(\m{X}^N, \m{U}^N)$ and an associated discrete costate
$\g{\Lambda}^N$ for which
\begin{equation}\label{maineq}
\max \left\{ \|{\bf X}^N-{\bf X}^*\|_\infty ,
\|{\bf U}^N-{\bf U}^*\|_\infty,
\|{\g \Lambda}^N-{\g \Lambda}^*\|_\infty \right\} \leq cN^{2-\eta},
\end{equation}
where $c$ is independent of $N$.
\end{theorem}
\smallskip

Although the discrete problem only possesses discrete controls at the
collocation points $-1 < \tau_i < +1$, $1 \le i \le N$, an estimate
for the discrete control at $t = -1$ and $t = +1$ is usually obtained
from the minimum principle (\ref{controlmin}) since we do have estimates
for the discrete state and costate at the end points.
Alternatively, polynomial interpolation could be used to obtain
estimates for the control at the end points of the interval.

The paper is organized as follows.
In Section~\ref{abstract} the discrete optimization
problem (\ref{D}) is reformulated as a nonlinear system of equations
obtained from the first-order optimality conditions,
and a general approach to convergence analysis is presented.
Section~\ref{residual} obtains an estimate for how closely the
solution to the continuous problem satisfies the first-order optimality
conditions for the discrete problem.
Section~\ref{inverse} proves that the linearization of the discrete
control problem around a solution of the continuous problem is invertible.
Section~\ref{omegabounds} establishes an $L^2$ stability property
for the linearization, while
Section~\ref{inftybounds} strengthens the norm to $L^\infty$.
This stability property is the basis for the proof of Theorem~\ref{maintheorem}.
A numerical example illustrating the exponential convergence
result is given in Section~\ref{numerical}.

{\bf Notation.}
The meaning of the norm $\| \cdot \|_\infty$ is based on context.
If $\m{x} \in \C{C}^0 (\mathbb{R}^n)$, then
$\|\m{x}\|_\infty$ denotes the maximum of $|\m{x}(t)|$ over
$t \in [-1, +1]$, where $| \cdot|$ is the Euclidean norm.
If $\m{A} \in \mathbb{R}^{m \times n}$, then $\|\m{A}\|_\infty$
is the largest absolute row sum (the matrix norm induces by the
$\ell_\infty$ vector norm).
If $\m{U} \in \mathbb{R}^{mN}$ is the discrete control with
$\m{U}_i \in \mathbb{R}^m$, then $\|\m{U}\|_\infty$ is the maximum
of $|\m{U}_i|$, $1 \le i \le N$.
The dimension of the identity matrix $\m{I}$ is often clear from context;
when necessary, the dimension of $\m{I}$ is specified by a subscript.
For example, $\m{I}_n$ is the $n$ by $n$ identity matrix.
$\nabla C$ denotes the gradient, a column vector, while
$\nabla^2 C$ denotes the Hessian matrix.
Throughout the paper, $c$ denotes a generic constant which has different
values in different equations.
The value of this constant is always independent of $N$.
$\m{1}$ denotes a vector whose entries are all equal to one, while
$\m{0}$ is a vector whose entries are all equal to zero, their
dimension should be clear from context.

\section{Abstract setting}
\label{abstract}
As shown in \cite{GargHagerRao10a}, the discrete problem (\ref{D})
can be reformulated as the nonlinear programming problem
\begin{equation}\label{nlp}
\begin{array}{ll}
\mbox {minimize} &C(\m{X}_{N+1})\\[.05in]
\mbox {subject to} &\sum_{j=0}^N{D}_{ij}{\m X}_j
={\m f}({\m X}_i,{\m U}_i), \quad 1\leq i\leq N, \quad \m{X}_0=\m{x}_0,\\[.05in]
&{\m X}_{N+1}={\m X}_0+\sum_{j=1}^N\omega_j{\m f}({\m X}_j,{\m U}_j).
\end{array}
\end{equation}
As indicated before Theorem~\ref{maintheorem}, $\m{X}_i$ corresponds to
$\m{x}^N(\tau_i)$.
Also, \cite{GargHagerRao10a} shows that the equations obtained by
setting the gradient of the Lagrangian to zero are equivalent
to the system of equations
\begin{eqnarray}
\sum_{j=1}^{N+1}{D}_{ij}^\dag{\g \Lambda}_j &=&
-\nabla_x H\left({\m X}_i,{\m U}_i, {\g \Lambda}_i\right), 
\quad 1 \leq i \leq N,  \quad
{\g \Lambda}_{N+1}=\nabla C(\m{X}_{N+1}), \label{Dadjoint} \\
\m{0} &=&
\nabla_u H\left({\m X}_i,{\m U}_i, {\g \Lambda}_i\right), 
\quad 1\leq i\leq N,  \label{Dcontrolmin}
\end{eqnarray}
where
\begin{eqnarray}
{D}_{ij}^\dag&=&-\displaystyle \left(\frac{\omega_j}{\omega_i}\right) {D}_{ji}, 
 \quad 1\leq i \leq N,\quad 1\leq j \leq N, \label{eq15} \\
{D}_{i,N+1}^\dag&=&-\sum_{j=1}^N{D}_{ij}^\dag, \quad
 1\leq i\leq N. \label{eq151}
\end{eqnarray}
Here $\g{\Lambda}_i$ corresponds to $\g{\lambda}^N(\tau_i)$.
The relationship between the discrete costate $\g{\Lambda}_i$, the
KKT multipliers $\g{\lambda}_i$ associated with the discrete dynamics,
and the multiplier $\g{\lambda}_{N+1}$ associated with the equation
for $\m{X}_{N+1}$ is
\begin{equation}\label{eq12} 
\omega_i {\g \Lambda}_i=\g{\lambda}_i+ \omega_i \g{\lambda}_{N+1}
\quad \mbox{when}\quad 1\leq i \leq N, \quad \mbox{ and} \quad
{\g \Lambda}_{N+1}=\g{\lambda}_{N+1}.
\end{equation}

The first-order optimality conditions for the nonlinear program (\ref{nlp})
consist of the equations (\ref{Dadjoint}) and (\ref{Dcontrolmin}),
and the constraints in (\ref{nlp}).
This system can be written as $\C{T}(\m{X}, \m{U}, \g{\Lambda}) = \m{0}$ where
\[
(\C{T}_1, \C{T}_2, \ldots, \C{T}_5) (\m{X}, \m{U}, \g{\Lambda}) \in
\mathbb{R}^{nN} \times \mathbb{R}^n \times
\mathbb{R}^{nN} \times \mathbb{R}^n \times \mathbb{R}^{mN} .
\]
The 5 components of $\C{T}$ are defined as follows:
\begin{eqnarray*}
\C{T}_{1i}(\m{X}, \m{U}, \g{\Lambda}) &=&
\left( \sum_{j=0}^N{D}_{ij}{\bf X}_j \right) -{\bf f}({\bf X}_i,{\bf U}_i), 
\quad 1\leq i\leq N ,\\
\C{T}_2(\m{X}, \m{U}, \g{\Lambda}) &=&
{\bf X}_{N+1}-{\bf X}_0-\sum_{j=1}^N\omega_j{\bf f}({\bf X}_j,{\bf U}_j),
\\
\C{T}_{3i}(\m{X}, \m{U}, \g{\Lambda}) &=&
\left( \sum_{j=1}^{N+1}{D}_{ij}^\dag{\g \Lambda}_j \right) +
\nabla_x H({\bf X}_i,{\bf U}_i, {\g \Lambda}_i),
\quad 1 \leq i \leq N , \\
\C{T}_4(\m{X}, \m{U}, \g{\Lambda}) &=&
{\g \Lambda}_{N+1}-\nabla_x C(\mathbf{X}_{N+1}), \\[.05in]
\C{T}_{5i}(\m{X}, \m{U}, \g{\Lambda}) &=&
\nabla_u H({\bf X}_i, {\bf U}_i, {\g \Lambda}_i),
\quad 1\leq i\leq N .
\end{eqnarray*}

Note that in formulating $\C{T}$, we treat $\m{X}_0$ as a constant
whose value is the given starting condition $\m{x}_0$.
Alternatively, we could treat $\m{X}_0$ as an unknown and then
expand $\C{T}$ to have a 6-th component $\m{X}_0 - \m{x}_0$.
With this expansion of $\C{T}$, we need to introduce an additional
multiplier $\g{\Lambda}_0$ for the constraint $\m{X}_0 - \m{x}_0$.
To achieve a slight simplification in the analysis, we employ a
5-component $\C{T}$ and treat $\m{X}_0$ as a constant, not an unknown.

The proof of Theorem~\ref{maintheorem} reduces to a study of solutions
to $\C{T}(\m{X}, \m{U}, \g{\Lambda}) = \m{0}$ in a neighborhood of
$(\m{X}^*, \m{U}^*, \g{\Lambda}^*)$.
Our analysis is based on 
\cite[Proposition 3.1]{DontchevHagerVeliov00}, which we simplify
below to take into account the structure of our $\C{T}$.
Other results like this are contained in
Theorem~3.1 of \cite{DontchevHager97},
in Proposition~5.1 of \cite{Hager99c}, and in Theorem~2.1 of \cite{Hager02b}.
\begin{proposition}\label{prop}
Let $\mathcal{X}$ be a Banach space and $\mathcal{Y}$ 
be a linear normed space with the norms in both spaces denoted 
$\|\cdot\|$. 
Let $\mathcal{T}$: $\mathcal{X} \longmapsto \mathcal{Y}$ with 
$\mathcal{T}$ continuously Fr\'{e}chet differentiable in 
$\C{B}_r(\g{\theta}^*)$ for
some $\g{\theta}^* \in \mathcal{X}$ and $r> 0$. 
Suppose that
\[
\|\nabla\mathcal{T}(\g{\theta})-\nabla\mathcal{T}(\g{\theta}^*)\|
\leq \varepsilon \mbox{ for all } \g{\theta} \in
{\mathcal B}_r(\g{\theta}^*)
\]
where $\nabla\mathcal{T}(\g{\theta}^*)$ is invertible, and
define $\mu := \|\nabla\mathcal{T}(\g{\theta}^*)^{-1}\|$.
If $\varepsilon\mu < 1$ and 
$\left\|\mathcal{T}\left(\g{\theta}^*\right)\right\|
\leq(1-\mu\varepsilon)r/\mu$, 
then there exists a unique 
$\g{\theta} \in {\mathcal B}_r(\g{\theta}^*)$ such that 
$\mathcal{T}(\g{\theta})=\m{0}$. 
Moreover, we have the estimate
\begin{equation}\label{abs}
\|\g{\theta}-\g{\theta}^*\|\leq \frac{\mu}{1-\mu\varepsilon}
\left\|\mathcal{T}\left(\g{\theta}^*\right)\right\| \le r.
\end{equation}
\end{proposition}

We apply Proposition~\ref{prop} with
$\g{\theta}^* = (\m{X}^*, \m{U}^*, \g{\Lambda}^*)$ and
$\g{\theta} = (\m{X}^N, \m{U}^N, \g{\Lambda}^N)$.
The key steps in the analysis are the estimation of the residual
$\left\|\mathcal{T}\left(\g{\theta}^*\right)\right\|$,
the proof that
$\nabla\mathcal{T}(\g{\theta}^*)$ is invertible,
and the derivation of a bound for
$\|\nabla\mathcal{T}(\g{\theta}^*)^{-1}\|$ that is independent of $N$.
In our context, for the norm in $\C{X}$, we take
\begin{equation}\label{Xnorm}
\|\g \theta\|=\|({\bf X},{\bf U}, {\g \Lambda})\|_\infty
=\max\{\|\bf X\|_\infty, \|\bf U\|_\infty,\|\g \Lambda\|_\infty\}.
\end{equation}
For this norm, the left side of (\ref{maineq}) and the left side of (\ref{abs})
are the same.
The norm on $\C{Y}$ enters into the estimation of both the residual
$\|\mathcal{T}(\g{\theta}^*)\|$ in
(\ref{abs}) and the parameter
$\mu := \|\nabla\mathcal{T}(\g{\theta}^*)^{-1}\|$.
In our context,
we think of an element of $\C{Y}$ as a vector with components
$\m{y}_i$, $1 \le i \le 3N + 2$,
where $\m{y}_i \in \mathbb{R}^n$ for $1 \le i \le 2N + 2$ and
$\m{y}_i \in \mathbb{R}^m$ for $i > 2N + 2$.
For example, $\C{T}_1(\m{X},\m{U}, \g{\Lambda}) \in \mathbb{R}^{nN}$
corresponds to the components $\m{y}_i \in \mathbb{R}^n$, $1 \le i \le N$.
For the norm in $\C{Y}$, we take
\begin{equation}\label{Ynorm}
\|\m{y}\|_\infty = \sup \{ |\m{y}_i| : 1 \le i \le 3N + 2 \} .
\end{equation}
%
\section{Analysis of the residual}
\label{residual}
We now establish a bound for the residual.
\smallskip
\begin{lemma}\label{residuallemma}
If {\rm (A1)} holds, then there
exits a constant $c$ independent of $N$ such that
\begin{equation}\label{delta}
\|\mathcal{T}(\m{X}^*, \m{U}^*, \g{\Lambda}^*)\|_\infty \le
c N^{2-\eta}
\end{equation}
for all $N > \eta+1$.
\end{lemma}
\begin{proof}
By the definition of $\C{T}$,
$\C{T}_4 (\m{X}^*,  \m{U}^*, \g{\Lambda}^*) = \m{0}$ since
$\m{x}^*$ and $\g{\lambda}^*$ satisfy the boundary condition in (\ref{costate}).
Likewise,
$\C{T}_5(\m{X}^*,  \m{U}^*, \g{\Lambda}^*) = \m{0}$ since
(\ref{controlmin}) holds for all $t \in [-1, +1]$;
in particular, (\ref{controlmin}) holds at the collocation points.

Now let us consider $\C{T}_1$.
By \cite[Eq. (7)]{GargHagerRao10a},
\[
\sum_{j=0}^N D_{ij} \m{X}_j^* = \dot{\m{x}}^I(\tau_i), \quad 1 \le i \le N,
\]
where $\m{x}^I \in \C{P}_N^n$ is the (interpolating) polynomial that passes
through $\m{x}^*(\tau_i)$ for $0 \le i \le N$.
Since $\m{x}^*$ satisfies the dynamics of (\ref{P}), it follows that
$\m{f}(\m{X}_i^*, \m{U}_i^*) = \dot{\m{x}}^* (\tau_i)$.
Hence, we have
\begin{equation}\label{T1i}
\C{T}_{1i}(\m{X}^*,  \m{U}^*, \g{\Lambda}^*) =
\dot{\m{x}}^I(\tau_i) - \dot{\m{x}}^*(\tau_i) .
\end{equation}
We combine Proposition~2.1 and Lemma~2.2
in \cite{HagerHouRao15a} to obtain for $N > \eta+1$,
\[
\|\dot{\m{x}}^I - \dot{\m{x}}^*\|_\infty  \le
\left( \frac{6e}{N-1} \right)^\eta
\left[ (1+2N^2) + 6e N (1+c_1 \sqrt{N}) \right]
\left( \frac{ 12 \|x^{(\eta+1)}\|}{\eta+1} \right) .
\]
where $\m{x}^{(\eta+1)}$ is the $(\eta+1)$-st derivative of $\m{x}$
and $c_1 \sqrt{N}$ is a bound for the Lebesgue constant of the
point set $\tau_i$, $0 \le i \le N$,
given in Theorem~4.1 of \cite{HagerHouRao15a}.
Hence, there exists a constant $c_2$, independent of $N$
but dependent on $\eta$, such that
\begin{equation} \label{h2}
\|\dot{\m{x}}^I - \dot{\m{x}}^*\|_\infty  \le
c_2 N^{2-\eta}.
\end{equation}
Consequently, $\C{T}_{1}(\m{X}^*,  \m{U}^*, \g{\Lambda}^*)$
complies with the bound (\ref{delta}).

Next, let us consider
\begin{equation}\label{h1}
\C{T}_2 (\m{X}^*,  \m{U}^*, \g{\Lambda}^*) =
\m{x}^* (1) - \m{x}^*(-1)
- \sum_{j=1}^N \omega_j \m{f}(\m{x}^*(\tau_j), \m{u}^*(\tau_j)) .
\end{equation}
By the fundamental theorem of calculus and the fact that
$N$-point Gauss quadrature is exact for polynomials of degree up to
$2N - 1$, we have
\begin{equation} \label{h3}
\m{0} =
{\bf x}^I(1)-{\bf x}^I(-1)-\int_{-1}^1\dot{\bf x}^I(t)dt =
{\bf x}^I(1)-{\bf x}^I(-1)-\sum_{j=1}^N\omega_j\dot{\bf x}^I(\tau_j) .
\end{equation}
Subtract (\ref{h3}) from (\ref{h1}) to obtain
\begin{equation}\label{h4}
\C{T}_2 (\m{X}^*,  \m{U}^*, \g{\Lambda}^*) =
{\bf x}^*(1)-{\bf x}^I(1)+\sum_{j=1}^N\omega_j\left(\dot{\bf x}^I(\tau_j)
-\dot{\bf x}^*(\tau_j)\right)
\end{equation}
Since $\omega_i > 0$ and their sum is 2, it follows (\ref{h2}) that
\begin{equation}\label{h5}
\left| \sum_{j=1}^N\omega_j\left(\dot{\bf x}^N(\tau_j)
-\dot{\bf x}^*(\tau_j)\right) \right| \le 2 c_2 N^{2-\eta} .
\end{equation}
By Theorem~15.1 in \cite{Trefethen13} and
Lemma~2.2 and Theorem~4.1 in \cite{HagerHouRao15a}, we have
\begin{eqnarray}
| {\bf x}^*(1)-{\bf x}^I(1)| &\le& \|\m{x}^* - \m{x}^I\|_\infty \nonumber \\
&\le& (1+ c_1 \sqrt{N}) \left( \frac{12}{\eta+2} \right)
\left( \frac{6e}{N} \right)^{\eta+1} \|\m{x}^{(\eta+1)}\|_\infty .
\label{h6}
\end{eqnarray}
We combine (\ref{h4})--(\ref{h6}) to see that
$\C{T}_2 (\m{X}^*,  \m{U}^*, \g{\Lambda}^*)$
complies with the bound (\ref{delta}).

Finally, let us consider $\C{T}_3$.
By \cite[Thm. 1]{GargHagerRao10a},
\[
\sum_{j=1}^{N+1} D_{ij}^\dagger \g{\Lambda}_j^* =
\dot{\g{\lambda}}^I(\tau_i), \quad 1 \le i \le N,
\]
where $\g{\lambda}^I \in \C{P}_N^n$ is the
(interpolating) polynomial that passes
through $\g{\Lambda}^*_j = \g{\lambda}(\tau_j)$ for $1 \le j \le N+1$.
Since $\g{\lambda}^*$ satisfies (\ref{costate}), it follows that
$\dot{\g{\lambda}}^*(\tau_i) =$
$-\nabla_x H({\bf X}_i^*,{\bf U}_i^*, {\g \Lambda}_i^*)$.
Hence, we have
\[
\C{T}_{3i}(\m{X}^*,  \m{U}^*, \g{\Lambda}^*) =
\dot{\g{\lambda}}^I(\tau_i) - \dot{\g{\lambda}}^*(\tau_i) .
\]
Exactly as we handled $\C{T}_1$ in (\ref{T1i}), we conclude that
$\C{T}_{3}(\m{X}^*,  \m{U}^*, \g{\Lambda}^*)$
complies with the bound (\ref{delta}).
This completes the proof.
\end{proof}

\section{Invertibility}
\label{inverse}
In this section, we show that the derivative $\nabla \C{T} (\g{\theta}^*)$
is invertible.
This is equivalent to showing that for each $\m{y} \in \C{Y}$,
there is a unique $\g{\theta} \in \C{X}$ such that
$\nabla \C{T} (\g{\theta}^*)[\g{\theta}] = \m{y}$.
In our application, $\g{\theta}^* = (\m{X}^*, \m{U}^*, \g{\Lambda}^*)$
and $\g{\theta} = (\m{X}, \m{U}, \g{\Lambda})$.
To simplify the notation, we let $\nabla \C{T}^*[\m{X}, \m{U}, \g{\Lambda}]$
denote the derivative of $\C{T}$ evaluated at
$(\m{X}^*, \m{U}^*, \g{\Lambda}^*)$ operating on $(\m{X}, \m{U}, \g{\Lambda})$.
This derivative involves the following 6 matrices:
\[
\begin{array}{ll}
{\bf A}_i=\nabla_x{\bf f}({\bf x}^*(\tau_i),{\bf u}^*(\tau_i)),
&{\bf B}_i=\nabla_u{\bf f}({\bf x}^*(\tau_i),{\bf u}^*(\tau_i)),\\
{\bf Q}_i=\nabla_{xx}H\left({\bf x}^*(\tau_i),{\bf u}^*(\tau_i),
{\g \lambda}^*(\tau_i)\right),
&{\bf S}_i=\nabla_{ux}H\left({\bf x}^*(\tau_i),{\bf u}^*(\tau_i),
{\g \lambda}^*(\tau_i)\right),\\
{\bf R}_i=\nabla_{uu}H\left({\bf x}^*(\tau_i),{\bf u}^*(\tau_i),
{\g \lambda}^*(\tau_i)\right), &{\bf T}=\nabla^2C({\bf x}^*(1)).
\end{array}
\]
With this notation,
the 5 components of $\nabla \C{T}^*[\m{X}, \m{U}, \g{\Lambda}]$ are as follows:
\begin{eqnarray*}
\nabla \C{T}_{1i}^*[\m{X}, \m{U}, \g{\Lambda}] &=&
\left( \sum_{j=1}^N{D}_{ij}{\bf X}_j \right)
- \m{A}_i \m{X}_i - \m{B}_i \m{U}_i,
\quad 1\leq i\leq N ,\\
\nabla \C{T}_2^*[\m{X}, \m{U}, \g{\Lambda}] &=&
{\bf X}_{N+1}-\sum_{j=1}^N\omega_j
(\m{A}_j \m{X}_j + \m{B}_j \m{U}_j),
\\
\nabla \C{T}_{3i}^*[\m{X}, \m{U}, \g{\Lambda}] &=&
\left( \sum_{j=1}^{N+1}{D}_{ij}^\dag{\g \Lambda}_j \right) +
\m{A}_i\tr \g{\Lambda}_i  + \m{Q}_i \m{X}_i + \m{S}_i \m{U}_i,
\quad 1 \leq i \leq N , \\
\nabla \C{T}_4^*[\m{X}, \m{U}, \g{\Lambda}] &=&
{\g \Lambda}_{N+1} -\m{T} \m{X}_{N+1}, \\[.05in]
\nabla \C{T}_{5i}^*[\m{X}, \m{U}, \g{\Lambda}] &=&
\m{S}_i\tr\m{X}_i + \m{R}_i \m{U}_i + \m{B}_i\tr \g{\Lambda}_i,
\quad 1\leq i\leq N .
\end{eqnarray*}
Notice that $\m{X}_0$ does not appear in $\nabla \C{T}^*$ since
$\m{X}_0$ is treated as a constant whose gradient vanishes.

The analysis of invertibility starts with the first component of
$\nabla \C{T}^*$.
\smallskip

\begin{lemma}
\label{feasible}
If {\rm (P1)} and {\rm (A3)} hold, then for each $\m{q} \in \mathbb{R}^n$
and $\m{p} \in \mathbb{R}^{nN}$ with $\m{p}_i \in \mathbb{R}^n$,
the linear system
\begin{eqnarray}
\left( \sum_{j=1}^{N}{D}_{ij}{\bf X}_j \right) - \m{A}_i \m{X}_i  &=& \m{p}_{i}
\quad 1\leq i\leq N , \label{h99} \\
{\bf X}_{N+1} -\sum_{j=1}^N\omega_j
(\m{A}_j \m{X}_j + \m{B}_j \m{U}_j) &=&{\m q}, \label{h100}
\end{eqnarray}
has a unique solution $\m{X}_j \in \mathbb{R}^{n}$,
$1 \le j \le N+1$.
This solution has the bound
\begin{equation}\label{xjbound}
\|\m{X}_j\|_\infty \le 4\|\m{p}\|_\infty + \| \m{q}\|_\infty, \quad
1 \le j \le N+1.
\end{equation}
\end{lemma}
\smallskip

\begin{proof}
Let $\bar{\m{X}}$ be the vector obtained by vertically stacking
$\m{X}_1$ through $\m{X}_N$,
let ${\m{A}}$ be the block diagonal matrix
with $i$-th diagonal block $\m{A}_i$, $1 \le i \le N$,
and define
$\bar{\m{D}} = {\bf D}_{1:N}\otimes {\bf I}_n$ where
$\otimes$ is the Kronecker product.
With this notation, the linear system (\ref{h99}) can be expressed
$(\bar{\m{D}} - {\m{A}}) \bar{\m{X}} = \m{p}$.
By (P1) ${\bf D}_{1:N}$ is invertible which implies that
$\bar{\m{D}}$ is invertible with
$\bar{\m{D}}^{-1} = {\bf D}_{1:N}^{-1}\otimes {\bf I}_n$.
Moreover,
$\|\bar{{\bf D}}^{-1}\|_\infty =$ $\|{\bf D}_{1:N}^{-1}\|_\infty \le 2$ by (P1).
By (A3) $\|{\m{A}}\|_\infty \le 1/4$ and
$\|\bar{\m{D}}^{-1}{\m{A}}\|_\infty \le$
$\|\bar{\m{D}}^{-1}\|_\infty \|{\m{A}}\|_\infty \le 1/2$.
By \cite[p. 351]{HJ12}, 
${\bf I}-\bar{\bf D}^{-1}{\bf A}$ is invertible and
$\left\|({\bf I}-{\bf D}^{-1}{\bf A})^{-1}\right\|_\infty \leq 2$.
Consequently, $\bar{\m{D}} - {\m{A}} =$
$\bar{\m{D}}(\m{I} - \bar{\m{D}}^{-1}{\m{A}})$ is invertible, and
\[
\|(\bar{\m{D}} - {\m{A}})^{-1}\|_\infty \le
\|(\m{I} - \bar{\m{D}}^{-1}{\m{A}})^{-1}\|_\infty
\|\bar{\m{D}}^{-1}\|_\infty \le 4.
\]
Thus there exists a unique $\bar{\m{X}}$ such that
$(\bar{\m{D}} - {\m{A}}) \bar{\m{X}} = \m{p}$, and
\begin{equation}\label{h100a}
\|\m{X}_j\|_\infty \le 4\|\m{p}\|_\infty , \quad
1 \le j \le N.
\end{equation}
By (\ref{h100}), we have
\begin{equation}\label{h100b}
\|\m{X}_{N+1}\|_\infty \le \|\m{q}\|_\infty + \sum_{j=1}^N
\omega_j \|\m{A}_j\|_\infty \|\m{X}_j\|_\infty .
\end{equation}
Since $\|\m{A}_j \|_\infty \le 1/4$ by (A3) and the $\omega_j$ are positive and
sum to 2, (\ref{h100a}) and (\ref{h100b}) complete the proof of (\ref{h100}).
\end{proof}
\smallskip

Next, we establish the invertibility of $\nabla \C{T}^*$.
\smallskip

\begin{proposition}
If {\rm (P1)}, {\rm (A2)} and {\rm (A3)} hold, then
$\nabla \C{T}^*$ is invertible.
\end{proposition}
\smallskip

\begin{proof}
We formulate a strongly convex quadratic programming
problem whose first-order optimality conditions reduce to
$\nabla \C{T}^*[\m{X}, \m{U}, \g{\Lambda}] = \m{y}$.
Due to the strong convexity of the objective function,
the quadratic programming has a solution and there exists
$\g{\Lambda}$ such that $\nabla \C{T}^*[\m{X}, \m{U}, \g{\Lambda}] = \m{y}$.
Since $\C{T}^*$ is square and
$\nabla \C{T}^*[\m{X}, \m{U}, \g{\Lambda}] = \m{y}$ has a solution
for each choice of $\m{y}$, it follows that $\nabla \C{T}^*$ is invertible.

The quadratic program is
\begin{equation}\label{QP}
\left.
\begin{array}{cl}
\mbox {minimize} &\frac{1}{2} \mathcal{Q}({\bf X},{\bf U})
+ \C{L}(\m{X}, \m{U}) \\[.08in]
\mbox {subject to} &\sum_{j=1}^N{D}_{ij}{\bf X}_j
={\bf A}_i{\bf X}_i+ {\bf B}_i{\bf U}_i+{\bf y}_{1i},
\quad 1\leq i \leq N, \\
&{\bf X}_{N+1}=\sum_{j=1}^N\omega_j
\left({\bf A}_j{\bf X}_j+{\bf B}_j{\bf U}_j
\right)+{\bf y}_2 ,
\end{array}
\right\}
\end{equation}
where the quadratic and linear terms in the objective are
\begin{eqnarray}
\mathcal{Q}({\bf X},{\bf U})&=&
{\bf X}_{N+1}\tr{\bf T}{\bf X}_{N+1}+\sum_{i=1}^N\omega_i
\left({\bf X}_i\tr{\bf Q}_i{\bf X}_i
+2{\bf X}_i\tr{\bf S}_i{\bf U}_i
+{\bf U}_i\tr{\bf R}_i{\bf U}_i\right), \label{Q}\\
\C{L}(\m{X}, \m{U}) &=& \m{y}_4\tr \m{X}_{N+1} - \sum_{i=1}^N\omega_i
\left( \m{y}_{3i}\tr \m{X}_i + \m{y}_{5i}\tr \m{U}_i \right) . \label{L}
\end{eqnarray}
The linear term was chosen so that the first-order optimality conditions
for (\ref{QP}) reduce to
$\nabla \C{T}^*[\m{X}, \m{U}, \g{\Lambda}] = \m{y}$.
See \cite{GargHagerRao10a} for the
manipulations needed to obtain the first-order optimality conditions
in this form.
By (A2), we have
\begin{equation}\label{lem1}
\mathcal{Q}({\bf X},{\bf U})\geq \alpha \left(
|{\bf X}_{N+1}|^2+\sum_{i=1}^N
\omega_i\left(|{\bf X}_i|^2 +|{\bf U}_i|^2\right) \right).
\end{equation}
Since $\alpha$ and $\g{\omega}$ are strictly positive, the objective of
(\ref{QP}) is strongly convex, and by Lemma~\ref{feasible}, the
quadratic programming problem is feasible.
Hence, there exists a unique solution to (\ref{QP}) for any choice of $\m{y}$,
and since the constraints are linear,
the first-order conditions hold.
Consequently,
$\nabla \C{T}^*[\m{X}, \m{U}, \g{\Lambda}] = \m{y}$ has a solution for
any choice of $\m{y}$ and the proof is complete.
\end{proof}

\section{$\omega$-norm bounds}
\label{omegabounds}
In this section we obtain a bound for the
$(\m{X}, \m{U})$ component of the solution to
$\nabla \C{T}^*[\m{X}, \m{U}, \g{\Lambda}] = \m{y}$ is terms of $\m{y}$.
The bound we derive in this section is in terms of the
$\omega$-norms defined by
\begin{equation}\label{onorm}
\|{\bf X}\|_\omega^2= |\m{X}_{N+1}|^2 +
\sum_{i=1}^N\omega_i|{\bf X}_i|^2 \quad \mbox{and} \quad
\|{\bf U}\|_\omega^2= \sum_{i=1}^N\omega_i|{\bf U}_i|^2.
\end{equation}
This defines a norm since the Gauss quadrature weight $\omega_i > 0$
for each $i$.
Since the
$(\m{X}, \m{U})$ component of the solution to
$\nabla \C{T}^*[\m{X}, \m{U}, \g{\Lambda}] = \m{y}$ is a solution
of the quadratic program (\ref{QP}), we will bound the solution
to the quadratic program.

First, let us think more abstractly.
Let $\pi$ be a symmetric, continuous bilinear functional defined on
a Hilbert space $\C{H}$, let $\ell$ be a continuous linear functional,
let $\phi \in \C{H}$, and consider the quadratic program
\[
\min \left\{ \frac{1}{2}\pi(v+\phi,v+\phi) + \ell (v+\phi):
v\in \C{V} \right\},
\]
where $\C{V}$ is a subspace of $\C{H}$.
If $w$ is a minimizer, then by the first-order optimality conditions,
we have
\[
\pi (w, v) + \pi (\phi, v) + \ell (v) = 0 \quad \mbox{for all } v \in \C{V}.
\]
Inserting $v = w$ yields
\begin{equation}\label{nullw}
\pi (w,w) = -(\pi (w, \phi) + \ell (w)).
\end{equation}

We apply this observation to the quadratic program (\ref{QP}).
We identify $\ell$ with the linear functional $\C{L}$ in (\ref{L}),
and $\pi$ with the bilinear form associated with the quadratic term (\ref{Q}).
The subspace $\C{V}$ is the null space of the linear operator
in (\ref{QP}) and $\phi$ is a particular solution of the linear system.
The complete solution of (\ref{QP}) is the particular solution plus
the minimizer over the null space.

In more detail, let $\g{\chi}$ denote the solution to (\ref{h99})--(\ref{h100})
given by Lemma~\ref{feasible} for $\m{p} = \m{y}_1$ and $\m{q} = \m{y}_2$.
We consider the particular solution $(\m{X}, \m{U})$
of the linear system in (\ref{QP}) given by $(\g{\chi},\m{0})$.
The relation (\ref{nullw}) describing the null space component $(\m{X}, \m{U})$
of the solution is
\begin{equation}\label{nullx}
\C{Q}(\m{X}, \m{U}) = - \left( (\g{\chi}_N+\m{y}_4)\tr \m{TX}_N
+ \sum_{i=1}^N \omega_i \left[ (\m{Q}_i \g{\chi}_i - \m{y}_{3i})\tr \m{X}_i -
\m{y}_{5i}\tr \m{U}_i \right] \right) .
\end{equation}
Here the terms containing $\g{\chi}$ are associated with $\pi (w, \phi)$,
while the remaining terms are associated with $\ell$ or equivalently,
with $\C{L}$.
By (A2) we have the lower bound
\begin{equation}\label{lower}
\C{Q}(\m{X}, \m{U}) \ge \alpha (\|\m{X}\|_\omega^2 + \|\m{U}\|_\omega^2) .
\end{equation}
All the terms on the right side of (\ref{nullx}) can be bounded with
the Schwarz inequality; for example,
\begin{eqnarray}
\sum_{i=1}^N \omega_i \m{y}_{3i}\tr \m{X}_i &\le&
\left( \sum_{i=1}^N \omega_i |\m{y}_{3i}|^2 \right)^{1/2}
\left( \sum_{i=1}^N \omega_i |\m{X}_{i}|^2 \right)^{1/2} \nonumber \\
&\le& \sqrt{2} \|\m{y}_3\|_\infty
\left( \|\m{X}\|_\omega^2 + \|\m{U}\|_\omega^2  \right)^{1/2} . \label{upper}
\end{eqnarray}
The last inequality exploits the fact that the $\omega_i$ sum to 2 and
$|\m{y}_{3i}| \le \|\m{y}_3\|_\infty$.
To handle the terms involving $\g{\chi}$ in (\ref{nullx}),
we utilize the upper bound
$\|\g{\chi}_j\|_\infty \le 5 \|\m{y}\|_\infty$ based on Lemma~\ref{feasible}
with $\m{p} = \m{y}_1$ and $\m{q} = \m{y}_2$.
Combining upper bounds of the form (\ref{upper}) with the lower bound
(\ref{lower}), we conclude from (\ref{nullx}) that both $\|\m{X}\|_\omega$ and
$\|\m{U}\|_\omega$ are bounded by a constant times $\|\m{y}\|_\infty$.
The complete solution of (\ref{QP}) is the null space component that we
just bounded plus the particular solution $(\g{\chi}, \m{0})$.
Again, since $\|\g{\chi}_j\|_\infty \le 5 \|\m{y}\|_\infty$,
we obtain the following result.
\smallskip

\begin{lemma}\label{l2}
If {\rm (A2)--(A3)} and {\rm (P1)} hold, then there exists a constant $c$,
independent of $N$, such that the solution $(\m{X}, \m{U})$
of $(\ref{QP})$ satisfies $\|\m{X}\|_\omega \le c \|\m{y}\|_\infty$ and
$\|\m{U}\|_\omega \le c \|\m{y}\|_\infty$.
\end{lemma}
\smallskip

\section{$\infty$-norm bounds}
\label{inftybounds}
We now need to convert these $\omega$-norm bounds for $\m{X}$ and
$\m{U}$ into $\infty$-norm bounds and at the same time, obtain an
$\infty$-norm estimate for $\g{\Lambda}$.
By Lemma~\ref{feasible}, the solution to the dynamics in (\ref{QP})
can be expressed
\begin{equation}\label{solvestate}
\bar{\m{X}} = (\m{I} - \bar{\m{D}}^{-1}{\m{A}})^{-1}
\bar{\m{D}}^{-1} \m{BU} + \m{p},
\end{equation}
where $\m{B}$ is the block diagonal matrix with $i$-th diagonal block $\m{B}_i$.
Taking norms and utilizing the bounds
$\|\m{p}\|_\infty \le 4 \|\m{y}_{1}\|_\infty$ and
$\left\|({\bf I}-\bar{\bf D}^{-1}{\bf A})^{-1}\right\|_\infty \leq 2$
from Lemma~\ref{feasible}, we obtain
\begin{equation}\label{h101}
\|\bar{\m{X}}\|_\infty \le
2 \|\bar{\m{D}}^{-1} \m{BU}\|_\infty + 4 \|\m{y}_1\|_\infty .
\end{equation}

We now write
\begin{equation}\label{h103}
\bar{\m{D}}^{-1} \m{BU} = [\m{D}_{1:N}^{-1}\otimes \m{I}_n] \m{BU} =
[(\m{W}^{1/2} \m{D}_{1:N})^{-1} \otimes \m{I}_n] \m{BU}_\omega ,
\end{equation}
where $\m{W}$ is the diagonal matrix with the quadrature weights on the
diagonal and $\m{U}_\omega$ is the vector whose $i$-th element is
$\sqrt{\omega_i} \m{U}_i$.
Note that the $\sqrt{\omega_i}$ factors in (\ref{h103}) cancel each other.
An element of the vector $\bar{\m{D}}^{-1} \m{BU}$ is the
dot product between a row of
$(\m{W}^{1/2} \m{D}_{1:N})^{-1} \otimes \m{I}_n$ and the column vector
$\m{BU}_\omega$.
By (P2) the rows of $(\m{W}^{1/2} \m{D}_{1:N})^{-1} \otimes \m{I}_n$ have
Euclidean length bounded by $\sqrt{2}$.
By the properties of matrix norms induced by vector norms, we have
\[
\|\m{BU}_\omega\|_2 \le \|\m{B}\|_2 \|\m{U}_\omega\|_2 =
\|\m{B}\|_2 \|\m{U}\|_\omega .
\]
It follows that
\begin{equation}\label{h102}
\|\bar{\m{D}}^{-1} \m{BU}\|_\infty \le \sqrt{2} \|\m{B}\|_2 \|\m{U}\|_\omega .
\end{equation}
Combine Lemma~\ref{l2} with (\ref{h101}) and (\ref{h102}) to deduce that
$\|\bar{\m{X}}\|_{\infty} \le c \|\m{y}\|_\infty$,
where $c$ is independent of $N$.
Since $|X_N| \le c\|\m{y}\|_\infty$ by Lemma~\ref{l2}, it follows that
$\|{\m{X}}\|_{\infty} \le c \|\m{y}\|_\infty$,

Next, we use the third and fourth components of the linear system
$\nabla \C{T}^*[\m{X}, \m{U}, \g{\Lambda}] = \m{y}$
to obtain bounds for $\g{\Lambda}$.
These equations can be written
\begin{equation}\label{c1}
\bar{\m{D}}^\dagger \bar{\g{\Lambda}} + \bar{\m{D}}_{N+1}^\dagger
\g{\Lambda}_{N+1} + {\m{A}}\tr \bar{\g{\Lambda}} +
{\m{Q}} \bar{\m{X}} + {\m{S}} {\m{U}} = \m{y}_3
\end{equation}
and
\begin{equation}\label{c2}
\g{\Lambda}_{N+1} - \m{TX}_{N+1} = \m{y}_4 ,
\end{equation}
where $ \bar{\g{\Lambda}}$ is obtained by vertically stacking
$\g{\Lambda}_1$ through $\g{\Lambda}_N$, ${\m{Q}}$ and
${\m{S}}$ are block diagonal matrices with $i$-th diagonal blocks
$\m{Q}_i$ and $\m{S}_i$ respectively,
$\bar{\m{D}}^\dagger = {\bf D}_{1:N}^\dagger\otimes {\bf I}_n$,
and $\bar{\m{D}}_{N+1}^\dagger = \m{D}_{N+1}^\dagger \otimes \m{I}_n$,
where $\m{D}_{N+1}^\dagger$ is the $(N+1)$-st column of $\m{D}^\dagger$.

We show in Proposition~\ref{dagger} of the Appendix that
$\m{D}_{1:N} = -\m{J} \m{D}_{1:N}^\dagger\m{J}$, where $\m{J}$ is the
exchange matrix with ones on its counterdiagonal and zeros elsewhere.
It follows that
$\m{D}_{1:N}^{-1} = -\m{J} (\m{D}_{1:N}^\dagger)^{-1}\m{J}$.
Consequently, the elements in 
$\m{D}_{1:N}^{-1}$ are the negative of the elements in
$(\m{D}_{1:N}^\dagger)^{-1}$, but rearranged.
As a result, $(\m{D}_{1:N}^\dagger)^{-1}$ also possesses
properties (P1) and (P2), and the analysis of the discrete costate
closely parallels the analysis of the state.
The main difference is that the costate equation contains the
additional $\g{\Lambda}_{N+1}$ term along with the additional
equation (\ref{c2}).
By (\ref{c2}) and the previously established bound
$\|{\m{X}}\|_{\infty} \le c \|\m{y}\|_\infty$, it follows that
\begin{equation}\label{c3}
\|\g{\Lambda}_{N+1}\|_{\infty} \le c \|\m{y}\|_\infty,
\end{equation}
where $c$ is independent of $N$.
Since $\m{D}^\dagger \m{1} = \m{0}$, we deduce that
$(\m{D}_{1:N}^\dagger)^{-1} \m{D}_{N+1} = -\m{1}$.
It follows that
\[
(\bar{\m{D}}^\dagger)^{-1}
\bar{\m{D}}_{N+1}^\dagger =
[({\bf D}_{1:N}^\dagger)^{-1}\otimes {\bf I}_n]
[\m{D}_{N+1}^\dagger \otimes \m{I}_n] = -\m{1}\otimes \m{I}_n .
\]
Exploiting this identity, the analogue of (\ref{solvestate}) is
\[
\bar{\g{\Lambda}} =
(\m{I} + (\bar{\m{D}}^\dagger)^{-1}{\m{A}}\tr)^{-1}
[(\m{1}\otimes \m{I}_n)\g{\Lambda}_{N+1} + (\bar{\m{D}}^\dagger)^{-1}
(\m{y}_3 - {\m{Q}} \bar{\m{X}} - {\m{S}} {\m{U}})] .
\]
Hence, we have
\[
\|\bar{\g{\Lambda}}\|_\infty \le
2\|(\m{1}\otimes \m{I}_n)\g{\Lambda}_{N+1} + (\bar{\m{D}}^\dagger)^{-1}
(\m{y}_3 - {\m{Q}} \bar{\m{X}} - {\m{S}} {\m{U}})\|_\infty .
\]
Moreover,
$\|(\m{1}\otimes \m{I}_n)\g{\Lambda}_{N+1} \|_\infty \le c\|\m{y}\|_\infty$
by (\ref{c3}) and
$\|(\bar{\m{D}}^\dagger)^{-1} \m{y}_3\|_\infty \le 2 \|\m{y}_3\|_\infty$.
The terms $\|(\bar{\m{D}}^\dagger)^{-1} {\m{Q}} \bar{\m{X}}\|_\infty$
and $\|(\bar{\m{D}}^\dagger)^{-1} {\m{S}} {\m{U}})\|_\infty$
are handled exactly as the term $\|\bar{\m{D}}^{-1} \m{BU}\|_\infty$ was handled
in the state equation (\ref{solvestate}).
We again conclude that 
$\|{\g{\Lambda}}\|_\infty \le c \|\m{y}\|_\infty$ where $c$ is independent
of $N$.

Finally, let us examine the fifth component of the linear system
$\nabla \C{T}^*[\m{X}, \m{U}, \g{\Lambda}] = \m{y}$.
These equations can be written
\[
\m{S}_i\tr\m{X}_i + \m{R}_i \m{U}_i + \m{B}_i\tr \g{\Lambda}_i = \m{y}_{5i},
\quad 1\leq i\leq N .
\]
By (A2) the smallest eigenvalue of $\m{R}_i$ is greater than $\alpha > 0$.
Consequently, the bounds
$\|\m{X}\|_\infty \le c \|\m{y}\|_\infty$ and
$\|\g{\Lambda}\|_\infty \le c \|\m{y}\|_\infty$ imply the existence
of a constant $c$, independent of $N$, such that
$\|\m{U}\|_\infty \le c \|\m{y}\|_\infty$.
In summary, we have the following result:

\begin{lemma}\label{inf-bounds}
If {\rm (A2)--(A3)} and {\rm (P1)--(P2)} hold,
then there exists a constant $c$,
independent of $N$, such that the solution
of $\nabla \C{T}^*[\m{X}, \m{U}, \g{\Lambda}] = \m{y}$ satisfies
\[
\|\m{X}\|_\infty + \|\m{U}\|_\infty + \|\g{\Lambda}\|_\infty \le
c \|\m{y}\|_\infty.
\]
\end{lemma}

Let us now prove Theorem~\ref{maintheorem} using Proposition~\ref{prop}.
By Lemma~\ref{inf-bounds},
$\mu = \|\nabla \C{T}(\m{X}^*, \m{U}^*, \g{\Lambda}^*)^{-1}\|_\infty$
is bounded uniformly in $N$.
Choose $\varepsilon$ small enough that $\varepsilon\mu < 1$.
When we compute the difference
$\nabla \C{T}(\m{X}, \m{U}, \g{\Lambda}) -
\nabla \C{T}(\m{X}^*, \m{U}^*, \g{\Lambda}^*)$ for
$(\m{X}, \m{U}, \g{\Lambda})$ near $(\m{X}^*, \m{U}^*, \g{\Lambda}^*)$
in the $\infty$-norm,
the $\m{D}$ and $\m{D}^\dagger$ constant terms cancel, and we are
left with terms involving the difference of
derivatives of $\m{f}$ or $C$ up to second order at nearby points.
By assumption, these second derivative are uniformly continuous on
the closure of $\Omega$ and on a ball around $\m{x}^*(1)$.
Hence, for $r$ sufficiently small, we have
\[
\
\|\nabla\mathcal{T}(\m{X}, \m{U}, \g{\Lambda})-
\nabla\mathcal{T}(\m{X}^*, \m{U}^*, \g{\Lambda}^*)\|_\infty
\leq \varepsilon
\]
whenever
\begin{equation}\label{r-bound}
\max\{\|\m{X} -\m{X}^*\|_\infty, \|\m{U} -\m{U}^*\|_\infty,
\|\g{\Lambda} - \g{\Lambda}^*\|_\infty\} \le r.
\end{equation}
By Lemma~\ref{residuallemma}, it follows that
$\left\|\mathcal{T}\left(\m{X}^*, \m{U}^*, \g{\Lambda}^*\right)\right\|
\leq(1-\mu\varepsilon)r/\mu$
for all $N$ sufficiently large.
Hence, by Proposition~\ref{prop}, there exists a solution to
$\C{T}(\m{X},\m{U}, \g{\Lambda}) = \m{0}$ satisfying (\ref{r-bound}).
Moreover, by (\ref{abs}) and (\ref{delta}), the estimate (\ref{maineq}) holds.
To complete the proof, we need to show that
$(\m{X}, \m{U})$ is a local minimizer for (\ref{nlp}).
After replacing the KKT multipliers by the transformed quantities given
by (\ref{eq12}), the Hessian of the Lagrangian is the following block
diagonal matrix:
\[
\mbox{diag} \left\{
\omega_1 \nabla_{(x,u)}^2 H(\m{X}_1, \m{U}_1, \g{\Lambda}_1),\;\; \ldots \;\;,
\;\; \omega_N \nabla_{(x,u)}^2 H(\m{X}_N, \m{U}_N, \g{\Lambda}_N), \;\;
\nabla^2 C(\m{X}_{N+1}) \right\}
\]
where $H$ is the Hamiltonian.
In computing the Hessian, we assume that the $\m{X}$ and $\m{U}$
variables are arranged in the following order:
$\m{X}_1$, $\m{U}_1$, $\m{X}_2$, $\m{U}_2$, $\ldots$, $\m{X}_N$, $\m{U}_N$,
$\m{X}_{N+1}$.
By (A2) the Hessian is positive definite when evaluated at
$(\m{X}^*, \m{U}^*, \g{\Lambda}^*)$.
By continuity of the second derivative of $C$ and $\m{f}$ and by the
convergence result (\ref{maineq}), we conclude that the Hessian of the
Lagrangian, evaluated at the solution of
$\C{T}(\m{X}, \m{U}, \g{\Lambda}) = \m{0}$ satisfying (\ref{r-bound}),
is positive definite for $N$ sufficiently large.
Hence, by the second-order sufficient optimality condition
\cite[Thm. 12.6]{NocedalWright2006}, $(\m{X},\m{U})$ is a strict
local minimizer of (\ref{nlp}).
This completes the proof of Theorem~\ref{maintheorem}.

\section{Numerical illustration}
\label{numerical}
Although the assumptions (A1)--(A3) are sufficient for exponential
convergence, the following example indicates that these assumptions
are conservative.
Let us consider the unconstrained control problem
\begin{equation}\label{exprob}
\min \left\{ -x(2) :
\dot{x}(t) = \textstyle{\frac{5}{2}}(-x(t) + x(t)u(t) - u(t)^2),
\; x(0) = 1 \right\} .
\end{equation}
The optimal solution and associated costate are
\begin{eqnarray*}
x^*(t) &=& 4/a(t), \quad a(t) = 1 + 3 \exp (2.5t), \\
u^*(t) &=& x^*(t)/2, \\
\lambda^* (t) &=& -a^2(t) \exp (-2.5t)/[\exp (-5) + 9 \exp(5) + 6].
\end{eqnarray*}
Figure~\ref{example} plots the logarithm of the sup-norm error in
the state, control, and costate as a function of the number of collocation
points.
Since these plots are nearly linear, the error behaves like
$c 10^{-\alpha N}$ where $\alpha \approx 0.6$ for either the state
or the control and $\alpha \approx 0.8$ for the costate.
In Theorem~\ref{maintheorem}, the dependence of the error on $N$
is somewhat complex due to the connection between $m$ and $N$.
As we increase $N$, we can also increase $m$ when the solution is
infinitely differentiable, however, the norm of the derivatives
also enters into the error bound as in (\ref{h2}).
Nonetheless, in cases where the solution derivatives can be bounded by
$c^m$ for some constant $c$, it is possible to deduce an exponential
decay rate for the error as observed in \cite[Sect. 2]{GargHagerRao10a}.
Note that the example problem (\ref{exprob}) does not satisfy (A2)
since $\nabla^2 C = \m{0}$, which is not positive definite.
Nonetheless, the pointwise error decays exponentially fast.
\begin{figure}
\begin{center}
\includegraphics[scale=.5]{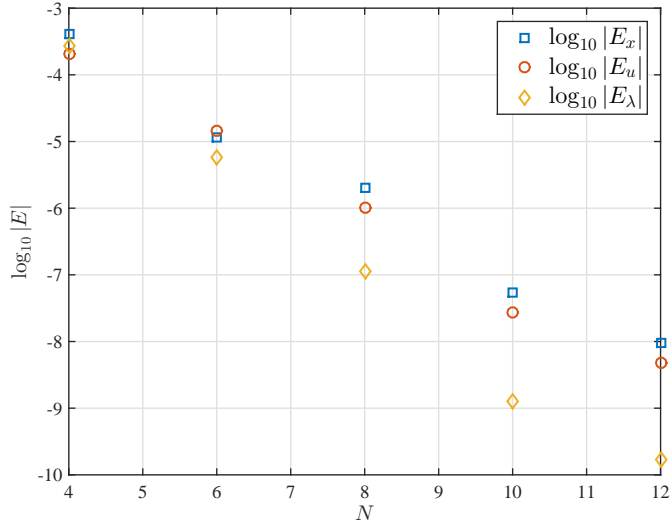}
\caption{The base 10 logarithm of the error in the sup-norm as a function of
the number of collocation points.
\label{example}}
\end{center}
\end{figure}

\section{Conclusions}
A Gauss collocation scheme is analyzed for an unconstrained control problem.
For a smooth solution whose Hamiltonian
satisfies a strong convexity assumption, we show that the discrete
problem has a local minimizer in a neighborhood of the continuous solution,
and as the number of collocation points increases, the distance in the
sup-norm between the discrete solution and the continuous solution is
$O(N^{2-\eta})$ when the continuous solution has $\eta+1$
continuous derivatives, $\eta \ge 3$,
and the number of collocation points $N$ is sufficiently large.
A numerical example is given which exhibits an exponential convergence rate.
\section{Appendix}
\label{appendix}
In (\ref{eq15}) we define a new matrix $\m{D}^\dagger$ in terms of
the differentiation matrix $\m{D}$.
The following proposition shows that the elements of
$\m{D}^\dagger$ are the negative and a rearrangement of the
elements of $\m{D}$.
\begin{proposition} \label{dagger}
The entries of the matrices $\m{D}$ and $\m{D}^\dagger$ satisfy
\[
D_{ij} = -D_{N+1-i, N+1-j}^\dagger,
\quad 1 \le i \le N, \quad 1 \le j \le N.
\]
In other words, $\m{D}_{1:N} = -\m{J} \m{D}_{1:N}^\dagger \m{J}$
where $\m{J}$ is the exchange matrix with ones on its counterdiagonal
and zeros elsewhere.
\end{proposition}
\smallskip
\begin{proof}
By (\ref{Ddef}) the elements of $\m{D}$ can be expressed in terms
of the derivatives of a set of Lagrange basis functions evaluated at
the collocation points:
\[
D_{ij} = \dot{L}_j(\tau_i) \quad \mbox{where } L_j \in \C{P}_N, \quad
L_j (\tau_k) = \left\{ \begin{array}{ll}
1 & \mbox{if } k = j, \\
0 & \mbox{if } 0 \le k \le N, \; k \ne j.
\end{array}
\right.
\]
In (\ref{Ddef}) we give an explicit formula for the Lagrange basis functions,
while here we express the basis function in terms of polynomials
$L_j$ that equal one at $\tau_j$ and vanish at $\tau_k$ where
$0 \le k \le N$, $k \ne j$.
These $N +1$ conditions uniquely define $L_j \in \C{P}_N$.
Similarly, by \cite[Thm. 1]{GargHagerRao10a},
the entries of $\m{D}_{1:N}^\dagger$ are given by
\[
D_{ij}^\dagger = \dot{M}_j(\tau_i) \quad \mbox{where } M_j \in \C{P}_N, \quad
M_j (\tau_k) = \left\{ \begin{array}{ll}
1 & \mbox{if } k = j, \\
0 & \mbox{if } 1 \le k \le N+1, \; k \ne j .
\end{array}
\right.
\]
Observe that $M_{N+1-j}(t) = L_j (-t)$ due the symmetry of the
quadrature points around $t = 0$:
\smallskip
\begin{itemize}
\item[(a)]
Since $-\tau_{N+1-j} = \tau_j$, we have
$L_j (-\tau_{N+1-j}) = L_j (\tau_j) = 1$.
\item[(b)]
Since $\tau_{N+1} = 1$ and $\tau_0 = -1$, we have
$L_j (-\tau_{N+1}) = L_j (\tau_0) = 0$.
\item[(c)]
Since $-\tau_i = \tau_{N+1-i}$, we have
$L_j (-\tau_{i}) = L_j (\tau_{N+1-i}) = 0$ if $i \ne N+1-j$.
\end{itemize}
Since $M_{N+1-j}(t)$ is equal to  $L_j (-t)$ at $N+1$ distinct points,
the two polynomials are equal everywhere.
Replacing $M_{N+1-j}(t)$ by $L_j (-t)$,
we have
\[
D_{N+1-i, N+1-j}^\dagger =
-\dot{L}_j (-\tau_{N+1-i}) = - \dot{L}_j (\tau_i) =-D_{ij}.
\]
\end{proof}
\smallskip

Tables \ref{P1} and \ref{P2} illustrate properties (P1) and (P2)
for the differentiation matrix $\m{D}$.
In Table~\ref{P1} we observe that $\|\m{D}_{1:N}^{-1}\|_\infty$
monotonically approaches the upper limit 2.
More precisely, it is found that $\|\m{D}_{1:N}^{-1}\|_\infty = 1 + \tau_N$,
where the final collocation point $\tau_N$ approaches one as $N$ tends
to infinity.
In Table~\ref{P2} we show the maximum 2-norm of the rows of
$[\m{W}^{1/2} \m{D}_{1:N}]^{-1}$.
It is found that the maximum is attained by the last row of
$[\m{W}^{1/2} \m{D}_{1:N}]^{-1}$, and the maximum monotonically approaches
$\sqrt{2}$.
\begin{table}[ht]
\centering
\begin{tabular}{c|c|c|c|c|c|c}
\hline\hline
$N$ & 25 & 50 & 75 & 100 & 125 & 150 \\
\hline
norm & 1.995557 & 1.998866 & 1.999494 & 1.999714 & 1.999816 & 1.999872  \\
\hline\hline
$N$ & 175 & 200 & 225 & 250 & 275 & 300 \\
\hline
norm & 1.999906 & 1.999928 & 1.999943 & 1.999954 & 1.999962 & 1.999968\\
\hline
\end{tabular}
\vspace*{.1in}
\caption{$\|{\bf D}_{1:N}^{-1}\|_\infty$
\label{P1}}
\end{table}
\begin{table}[ht]
\centering
\begin{tabular}{c|c|c|c|c|c|c}
\hline\hline
$N$ & 25 & 50 & 75 & 100 & 125 & 150 \\
\hline
norm & 1.412201 & 1.413703 & 1.413985 & 1.414085 & 1.414131 & 1.414156 \\
\hline\hline
$N$ & 175 & 200 & 225 & 250 & 275 & 300 \\
\hline
norm & 1.414171 & 1.414181 & 1.414188 & 1.414193 & 1.414196 & 1.414199 \\
\hline
\end{tabular}
\vspace*{.1in}
\caption{Maximum Euclidean norm for the rows of
$[\m{W}^{1/2}{\bf D}_{1:N}]^{-1}$
\label{P2}}
\end{table}
%


\newpage
\bibliographystyle{siam}
\bibliography{library}

\begin{thebibliography}{10}

\bibitem{Benson2}
{\sc D.~A. Benson, G.~T. Huntington, T.~P. Thorvaldsen, and A.~V. Rao}, {\em
  Direct trajectory optimization and costate estimation via an orthogonal
  collocation method}, J. Guid. Control Dyn., 29 (2006), pp.~1435--1440.

\bibitem{DontchevHager93}
{\sc A.~L. Dontchev and W.~W. Hager}, {\em Lipschitzian stability in nonlinear
  control and optimization}, {SIAM} J. Control Optim., 31 (1993), pp.~569--603.

\bibitem{DontchevHager97}
\leavevmode\vrule height 2pt depth -1.6pt width 23pt, {\em The {Euler}
  approximation in state constrained optimal control}, Math. Comp., 70 (2001),
  pp.~173--203.

\bibitem{DontchevHagerMalanowski00}
{\sc A.~L. Dontchev, W.~W. Hager, and K.~Malanowski}, {\em Error bounds for
  {Euler} approximation of a state and control constrained optimal control
  problem}, Numer. Funct. Anal. Optim., 21 (2000), pp.~653--682.

\bibitem{DontchevHagerVeliov00}
{\sc A.~L. Dontchev, W.~W. Hager, and V.~M. Veliov}, {\em Second-order
  {Runge}-{Kutta} approximations in constrained optimal control}, {SIAM} J.
  Numer. Anal., 38 (2000), pp.~202--226.

\bibitem{Elnagar1}
{\sc G.~Elnagar, M.~Kazemi, and M.~Razzaghi}, {\em The pseudospectral
  {Legendre} method for discretizing optimal control problems}, IEEE Trans.
  Automat. Control, 40 (1995), pp.~1793--1796.

\bibitem{Elnagar4}
{\sc G.~N. Elnagar and M.~A. Kazemi}, {\em Pseudospectral {Chebyshev} optimal
  control of constrained nonlinear dynamical systems}, Comput. Optim. Appl., 11
  (1998), pp.~195--217.

\bibitem{Fahroo2}
{\sc F.~Fahroo and I.~M. Ross}, {\em Costate estimation by a {Legendre}
  pseudospectral method}, J. Guid. Control Dyn., 24 (2001), pp.~270--277.

\bibitem{FahrooRoss02}
\leavevmode\vrule height 2pt depth -1.6pt width 23pt, {\em Direct trajectory
  optimization by a {Chebyshev} pseudospectral method}, J. Guid. Control Dyn.,
  25 (2002), pp.~160--166.

\bibitem{Fahroo3}
\leavevmode\vrule height 2pt depth -1.6pt width 23pt, {\em Pseudospectral
  methods for infinite-horizon nonlinear optimal control problems}, J. Guid.
  Control Dyn., 31 (2008), pp.~927--936.

\bibitem{GargHagerRao11a}
{\sc D.~Garg, M.~A. Patterson, C.~L. Darby, C.~Fran\c{c}olin, G.~T. Huntington,
  W.~W. Hager, and A.~V. Rao}, {\em Direct trajectory optimization and costate
  estimation of finite-horizon and infinite-horizon optimal control problems
  using a {Radau} pseudospectral method}, Comput. Optim. Appl., 49 (2011),
  pp.~335--358.

\bibitem{GargHagerRao10a}
{\sc D.~Garg, M.~A. Patterson, W.~W. Hager, A.~V. Rao, D.~A. Benson, and G.~T.
  Huntington}, {\em A unified framework for the numerical solution of optimal
  control problems using pseudospectral methods}, Automatica, 46 (2010),
  pp.~1843--1851.

\bibitem{GongRossKangFahroo08}
{\sc Q.~Gong, I.~M. Ross, W.~Kang, and F.~Fahroo}, {\em Connections between the
  covector mapping theorem and convergence of pseudospectral methods for
  optimal control}, Comput. Optim. Appl., 41 (2008), pp.~307--335.

\bibitem{Hager99c}
{\sc W.~W. Hager}, {\em {Runge}-{Kutta} methods in optimal control and the
  transformed adjoint system}, Numer. Math., 87 (2000), pp.~247--282.

\bibitem{Hager02b}
\leavevmode\vrule height 2pt depth -1.6pt width 23pt, {\em Numerical analysis
  in optimal control}, in International Series of Numerical Mathematics, K.-H.
  Hoffmann, I.~Lasiecka, G.~Leugering, J.~Sprekels, and F.~Tr\"{o}ltzsch, eds.,
  vol.~139, Basel/Switzerland, 2001, Birkhauser Verlag, pp.~83--93.

\bibitem{HagerHouRao15a}
{\sc W.~W. Hager, H.~Hou, and A.~V. Rao}, {\em Lebesgue constants arising in a
  class of collocation methods}, IMA J.~Numer.~Anal., submitted (2015,
  arxiv.org/abs/1507.08316).

\bibitem{HJ12}
{\sc R.~A. Horn and C.~R. Johnson}, {\em Matrix Analysis}, Cambridge University
  Press, Cambridge, 2013.

\bibitem{Kameswaran1}
{\sc S.~Kameswaran and L.~T. Biegler}, {\em Convergence rates for direct
  transcription of optimal control problems using collocation at radau points},
  Comput. Optim. Appl., 41 (2008), pp.~81--126.

\bibitem{LiuHagerRao15}
{\sc F.~Liu, W.~W. Hager, and A.~V. Rao}, {\em Mesh refinement for optimal
  control using nonsmoothness detection and mesh size reduction}, J. Franklin
  Inst.,  (2015, to appear).

\bibitem{NocedalWright2006}
{\sc J.~Nocedal and S.~J. Wright}, {\em Numerical Optimization}, Springer, New
  York, 2nd~ed., 2006.

\bibitem{PattersonHagerRao14}
{\sc M.~A. Patterson, W.~W. Hager, and A.~V. Rao}, {\em A $ph$ mesh refinement
  method for optimal control}, Optim. Control Appl. Meth., 36 (2015),
  pp.~398--421.

\bibitem{Reddien79}
{\sc G.~W. Reddien}, {\em Collocation at {Gauss} points as a discretization in
  optimal control}, {SIAM} J. Control Optim., 17 (1979), pp.~298--306.

\bibitem{Trefethen13}
{\sc L.~N. Trefethen}, {\em Approximation Theory and Approximation Practice},
  SIAM Publications, Philadelphia, 2013.

\bibitem{Williams1}
{\sc P.~Williams}, {\em Jacobi pseudospectral method for solving optimal
  control problems}, J. Guid. Control Dyn., 27 (2004), pp.~293--297.

\end{thebibliography}
\end{document}